\documentclass[reqno, 12pt]{amsart}

\usepackage{cite}
\usepackage{graphicx}
\usepackage[a4paper,centering, scale=0.8]{geometry}
\date{}

\allowdisplaybreaks[4]
\theoremstyle{plain}
\newtheorem{thm}{Theorem}[section]

\newtheorem{lem}[thm]{Lemma}
\newtheorem{cor}[thm]{Corollary}

\theoremstyle{remark}
\newtheorem{rem}{Remark}[section]

\theoremstyle{definition}

\usepackage[colorlinks,linkcolor=blue,urlcolor=red,linktocpage]{hyperref}

\numberwithin{equation}{section}

\begin{document}
\title
{a Pr\'{e}kopa-Leindler type inequality related to the $L_p$  Brunn-Minkowski inequality}

\author[Y. Wu]{Yuchi Wu}
\address[Y.C. Wu]{School of Mathematics science, East China Normal University, Shanghai 200241, China; Shanghai Key Laboratory of Pure Mathematics and Mathematical Practice, Shanghai 200241, China}
 \email{\href{mailto: Yuchi Wu
<wuyuchi1990@126.com>}{wuyuchi1990@126.com}}

\begin{abstract}
In this paper, we prove a Pr\'{e}kopa-Leindler type inequality related to the $L_p$ Brunn-Minkowski inequality. It extends an inequality proved by Das Gupta \cite{MR588074} and  Klartag \cite{MR2349606}, and thus  recovers the Pr\'{e}kopa-Leindler inequality. In addition, we prove a functional $L_p$ Minkowski inequality.
\end{abstract}



\thanks{The author is supported by Project funded by China Postdoctoral Science Foundation 2019TQ0097, NSFC 11671249 and  a research grant from Shanghai Key Laboratory of PMMP 18dz2271000.}

\maketitle

\section{Introduction}

One of the cornerstones of the Brunn-Minkowski theory is the celebrated Brunn-Minkowski inequality (see, e.g., the books by Gardner \cite{MR2251886}, Gruber \cite{MR2335496} and Schneider \cite{MR3155183} for references). It has had far reaching consequences for subjects quite distant from geometric convexity. For this, see the wonderful survey by Gardner \cite{MR1898210}. By the middle of the last century, the Brunn-Minkowski inequality had been successfully extended to nonconvex sets.

\begin{thm}[the Brunn-Minkowski inequality]\label{TBMI}Let $K$ and $L$ be nonempty bounded measurable
sets in $n$-dimensional Euclidean space $\mathbb{R}^{n}$ such that $(1-\lambda)K + \lambda L$ is also measurable. Then
\begin{equation}\label{BMI}
V_{n}((1-\lambda)K + \lambda L)^{\frac{1}{n}} \geq (1-\lambda)V_{n}(K)^{\frac{1}{n}}+\lambda V_{n}(L)^{\frac{1}{n}}.
\end{equation}
\end{thm}
Here, $V_{n}$ denotes the $n$-dimensional Lebesgue measure and $K+L=\{x+y: x \in$ $K, y \in L\}$ is the Minkowski sum of $K$ and $L$.

Denote by $\int  f$ the integral of a function $f$ on its domain with respect to the Lebesgue measure. The following Pr\'{e}kopa-Leindler inequality \cite{MR404557} is a functional type of the Brunn-Minkowski inequality (\ref{BMI}), see Pr\'{e}kopa \cite{MR315079,MR404557,MR480923} and Leindler \cite{MR2199372}.

\begin{thm}[the Pr\'{e}kopa-Leindler inequality]\label{PLI}
 Let $0<\lambda<1$ and let $f, g$ and $h$ be nonnegative integrable functions on $\mathbb{R}^{n}$ satisfying
\[
h((1-\lambda) x+\lambda y) \geq f(x)^{1-\lambda} g(y)^{\lambda}
\]
for all $x, y \in \mathbb{R}^{n} .$ Then
\[
\int  h   \geq\left(\int  f  \right)^{1-\lambda}\left(\int  g  \right)^{\lambda}.
\]
\end{thm}
The Pr\'{e}kopa-Leindler inequality can quickly imply the Brunn-Minkowksi inequality (\ref{BMI}), see section 7 in \cite{MR1898210} for details. This connection helps trigger a fruitful development of functional analogues of several geometric parameters into the class of log-concave functions currently undergoing, see, e.g., \cite{MR3554706,MR3790496,MR932007,MR3147446,MR3886184,MR2163897,MR2386937,MR3695895,MR3077887,MR3837280}.

The following Borell-Brascamp-Lieb inequality \cite{MR404559,MR0450480} generalizes the Pr\'{e}kopa-Leindler inequality, which is just the case $ \alpha= 0$. %

\begin{thm}[the Borell-Brascamp-Lieb inequality]\label{BBLI} Let $0<\lambda<1,$ let $-1 / n \leq \alpha \leq \infty,$ and let $f, g,$ and $h$ be nonnegative integrable functions on $\mathbb{R}^{n}$ satisfying
\begin{equation}\label{BBLIcondition}
h((1-\lambda) x+\lambda y) \geq M_{\alpha}(f(x), g(y), \lambda)\end{equation}
for all $x, y \in \mathbb{R}^{n} .$ Then
\begin{equation}\label{BBLIconclusion}\int  h   \geq M_{\alpha /(n \alpha+1)}\left(\int  f  , \int  g  , \lambda\right)\end{equation}\end{thm}
Here, for $\lambda\in (0,1)$, the $\alpha$-{\it means} $M_{\alpha}(a, b,{\lambda })$ of $a,b\geq 0$ is defined by
$$M_{\alpha}(a, b,{\lambda })=\left\{\begin{array}{ll}
\left((1-\lambda ) a^{\alpha}+\lambda  b^{\alpha}\right)^{1 / \alpha} & \text { if } \alpha \neq 0, \pm \infty \\
a^{1-\lambda } b^{\lambda } & \text { if } \alpha=0 \\
\max \{a, b\} & \text { if } \alpha=\infty \\
\min \{a, b\} & \text { if } \alpha=-\infty
\end{array}\right.$$
if $ab>0$, and $M_{\alpha}^{\lambda }(a, b)=0$ if $ab=0.$

Firey \cite{MR141003} generalizes Minkowski addition to $L_p$ addition for convex bodies (compact, convex subsets with nonempty interiors) containing the origin and proves the $L_p$ Brunn-Minkowski inequality, see also Lutwak \cite{MR1231704}.

If $K,L$ are two~convex bodies containing the origin, then the {\it $L_p$ sum} $K+_pL$ of $K$ and $L$ is defined by
\begin{equation}\label{lpsumconvexbody}
h_{K+_pL}(x)^p = h_K(x)^p + h_L(x)^p
\end{equation}
for all $x \in\mathbb{R}^n$, where $h_A(x)=\max\{x\cdot y:y\in A\}$ is the support function of $A$. Here, we denote by ``$\cdot$" the standard scalar product.

Lutwak, Yang and Zhang \cite{MR2873885} provided the {\it explicit pointwise formula of $L_p$ addition}:
\begin{equation}\label{pointwiselpsumconvexbody}
K+_{p} L=\left\{(1-\lambda)^{1 / q} x+\lambda^{1 / q} y: x \in K, y \in L, 0 \leq \lambda \leq 1\right\}
\end{equation}
where $q$ is the H\"{o}lder conjugate of $p,$ i.e. $\frac{1}{p}+\frac{1}{q}=1 .$ When $p=1, q=\infty,$ and $1 / q$ is defined as $0.$ It is worth to point out that the pointwise $L_p$ addition (\ref{pointwiselpsumconvexbody}) is suited for nonconvex sets.

Lutwak, Yang and Zhang \cite{MR2873885} established the following $L_p$ Brunn-Minkowski inequality for compact sets:~

\begin{thm}[Theorem 4, \cite{MR2873885}]\label{TLpBMI}
  Suppose $p \geq 1$. If $K$ and $L$ are nonempty compact sets in $\mathbb{R}^n$, then
\begin{equation}\label{LpBMI}  V_n(K+_pL)^\frac{p}{n}\geq V_n(K)^\frac{p}{n}+V_n(L)^\frac{p}{n}.\end{equation}
\end{thm}

We consider the following problem: is there a Pr\'{e}kopa-Leindler type inequality that can simply imply the $L_p$ Brunn-Minkowski inequality (\ref{LpBMI})?

For any nonnegative function $f$ on $\mathbb{R}^n$, we denote $\operatorname{supp}f$ by the support of $f$, i.e., the closure of $\{x\in\mathbb{R}^n:f(x)>0\}$. In this paper, inspired by Klartag \cite{MR2349606}, we will prove the following theorem:

\begin{thm}\label{LpBMs}
  Let $p\geq 1,$ let $s,\mu,\omega>0,$ and let $f,g,h:\mathbb{R}^n\rightarrow [0,\infty)$ be integrable with nonepmty supports. If for all $x\in \operatorname{supp}f,y\in\operatorname{supp}g$ and $\lambda\in [0,1],$
  \begin{equation}\label{FLpBMIcondition}h((1-\lambda)^{\frac{1}{q}}\mu^\frac{1}{p} x+\lambda^{\frac{1}{q}}\omega^\frac{1}{p} y)^{\frac{1}{s}}\geq (1-\lambda)^{\frac{1}{q}}\mu^\frac{1}{p} f(x)^{\frac{1}{s}}+\lambda^{\frac{1}{q}}\omega^\frac{1}{p} g(y)^{\frac{1}{s}},\end{equation}
  where $q$ is the {H\"{o}lder} conjugate of $p$.
  Then,
  \begin{equation}\label{FLpBMIconclusion}
\left(\int  h\right)^{\frac{p}{n+s}} \geq \mu\left(\int  f\right)^{\frac{p}{n+s}}+\omega\left(\int  g\right)^{\frac{p}{n+s}}.
\end{equation}
\end{thm}

Let $\mu=\omega=1$ and $f=\chi_K,g=\chi_ L,h=\chi_{K+_p L}$, where $\chi _E$ denotes the
characteristic function of $E$. It follows from (\ref{pointwiselpsumconvexbody}) that (\ref{FLpBMIcondition}) holds. Thus, we obtain (\ref{LpBMI}) by letting $s\rightarrow 0^+$ in (\ref{FLpBMIconclusion}). This shows that Theorem \ref{LpBMs} can be viewed as a functional generalization of Theorem \ref{TLpBMI}.

Theorem \ref{LpBMs} for $p=1$ is proved by Das Gupta \cite{MR588074} and Klartag \cite{MR2349606}, which recovers the Pr\'{e}kopa-Leindler inequality, see Corollary 2.2 in \cite{MR2349606} for details.

Let $p=1$, $s=\frac{1}{\alpha}$ and $\mu+\omega=1$ in Theorem \ref{LpBMs}. When $f,g$ are positive in their supports and have positive integrals, Theorem \ref{LpBMs} and Theorem \ref{BBLI} coincide with each other for $\alpha\in(0,\infty)$.

This paper was completed and sumbitted to a Journal on February 18, 2020. We recently found that \cite{RoysdonXing} was published on Arxiv on April 20, 2020. Letting $\mu+\omega=1 $ in Theorem \ref{LpBMs}, we can obtain Theorem 2.1 of \cite{RoysdonXing}.

This paper is organized as follows: In section 2, some basic facts and definitions for quick
reference are provided. In section 3, some useful lemmas are given. In Section 4, we prove Threorem \ref{LpBMs} and a functional $L_p$
Minkowski inequality.

 \vskip 20pt
\section{Preliminaries}

In this section, we collect some terminologies and notations. We recommend  the
books of Gardner \cite{MR2251886}, Gruber \cite{MR2335496} and Schneider \cite{MR3155183} as excellent references on convex geometry.

For a nonempty set $M\subset \mathbb{R}^2$, the {\it $M$-addition} of two sets $K,L\subset \mathbb{R}^n$ is defined as
$$K\oplus_ML=\{ax+by:(a,b)\in M, x\in K,y\in L\},$$
see Protasov \cite{MR1486700}, Gardner,  Hug and Weil \cite{MR3120744}, and Mesikepp\cite{MR3508484}. If $M=\{(1,1)\},$ then $M$-addition is the classical Minkowski addition. For $p\geq 1$ and its H\"{o}lder conjugate $q$, if $M=\{(a,b):a^q+b^q=1,a\geq0,b\geq0\},$ then $M$-addition is the explicit pointwise formula of $L_p$ addition (\ref{pointwiselpsumconvexbody}).

For $0<s<\infty$, we say that $f:\mathbb{R}^n
\rightarrow [0,\infty)$ is {\it$s$-concave} if $\operatorname{supp}(f)$ is nonempty, compact and convex, and $f^\frac{1}{s}$ is concave, i.e., for all $x,y\in \operatorname{supp}f$ and $0\leq \lambda\leq 1$, we have
$$f(\lambda x+(1-\lambda) y) \geq\left[\lambda f(x)^{\frac{1}{s}}+(1-\lambda) f(y)^{\frac{1}{s}}\right]^s.$$
$s$-concave function has been studied by Avriel \cite{MR301151},  Borell \cite{MR404559}, Brascamp and Lieb \cite{MR404559,MR0450480}, Rotem \cite{MR3062743,MR3253780}. 
For any function $f:\mathbb{R}^n\rightarrow [0,\infty)$ and any integer $s>0$, Klartag \cite{MR2349606} defines
\begin{equation}\label{defitionKf}
\mathcal{K}_f=\left\{(x,y)\in\mathbb{R}^{n+s}=\mathbb{R}^n\times\mathbb{R}^s:x
\in\operatorname{supp}f,|y|\leq f(x)^\frac{1}{s}\right\}.
\end{equation}
where, for given $x \in\mathbb{R}^n$ and $y \in\mathbb{R}^s$, $(x, y)$ are coordinates in $\mathbb{R}^{n+s}$. 
This set $\mathcal{K}_f$ is nonempty and convex if and only if $f$ is $s$-concave.
The volume of $\mathcal{K}_f$ can be computed as
\begin{equation}\label{Kfvolume}
{V}_{n+s}\left(\mathcal{K}_{f}\right)=\int_{\operatorname{supp}f} \kappa_{s} \cdot\left(f^{\frac{1}{s}}(x)\right)^{s}  =\kappa_{s} \int f,
\end{equation}
where $\kappa_{s}=\frac{\pi^{s / 2}}{\Gamma\left(\frac{s}{2}+1\right)}$ is the volume of the $s$-dimensional Euclidean unit ball.

For positive number $s$, two functions $f,g:\mathbb{R}^n\rightarrow [0,\infty)$ with nonempty supports and nonempty set $M\subset \mathbb{R}^2$ with nonnegative coordinates, we define the function $f \oplus_{M,s} g$ as
$$\left[f \oplus_{M,s} g\right](z)=\sup \left\{\left(af(x)^{\frac{1}{s}}+bg(y)^{\frac{1}{s}}\right)^s:x\in \operatorname{supp}(f), y \in \operatorname{supp}(g),\\z=ax+by,(a,b)\in M \right\}$$
when $z\in \operatorname{supp}f\oplus_M\operatorname{supp}g.$ If $z\notin \operatorname{supp}f\oplus_M\operatorname{supp}g$, we set
$\left[f \oplus_{M,s} g\right](z)=0$.
This definition is motivated by Gardner and Kiderlen \cite{MR2349606}, and Klartag \cite{MR3853928} .

For $s>0$, two functions $f,g:\mathbb{R}^n\rightarrow [0,\infty)$ with nonempty supports, $p\geq 1$ and its  H\"{o}lder conjugate $q$, we define
$f \oplus_{p,s} g$ as $f \oplus_{M,s} g$ by taking $M=\{(a,b):a^q+b^q=1,a,b\geq0\},$ i.e.,
$$\begin{aligned}\left[f \oplus_{p,s} g\right](z)=\sup \Big\{&\left((1-\lambda)^\frac{1}{q}f(x)^{\frac{1}{s}}+\lambda^\frac{1}{q}g(y)^{\frac{1}{s}}\right)^s:x \in \operatorname{supp}(f), y \in \operatorname{supp}(g),\\&\lambda\in[0,1], z=(1-\lambda)^\frac{1}{q}x+\lambda^\frac{1}{q}y\Big \}\end{aligned}$$
when $z\in \operatorname{supp}f+_p\operatorname{supp}g.$ If $z\notin \operatorname{supp}f+_p\operatorname{supp}g$, we set
$\left[f \oplus_{p,s} g\right](z)=0$.

For $s>0$, $p\geq 1$,  $\lambda > 0$ and $f : \mathbb{R}^n\rightarrow [0,\infty)$, we define the function $\lambda \times_{p,s} f : \mathbb{R}^n \rightarrow [0,\infty)$ as
\begin{equation}\label{lambdatimes}
[\lambda  \times_{p,s} f] (x) = \lambda ^\frac{s}{p}f({\lambda^{-\frac{1}{p}}}{x}).
\end{equation}
Note that, condition (\ref{FLpBMIcondition}) implies
\begin{equation}\label{hgeqf+g}
h\geq [\mu\times_{p,s}f]\oplus_{p,s} [\omega \times_{p,s}g]
\end{equation}
pointwise.

If $s$ is an integer, it is easy to see that $\mathcal{K}_{\lambda \times_{p,s}f} = \lambda^\frac{1}{p} \mathcal{K}_f = \{\lambda^\frac{1}{p} y: y \in \mathcal{K}_f \}$. Thus,
\begin{equation}\label{lambdatimesv}
V_{n+s}(\mathcal{K}_{\lambda \times_{p,s}f}) = \lambda ^\frac{n+s}{p}V_{n+s}(\mathcal{K}_f).
\end{equation}

\section{some useful lemmas
}

Let $p\geq 1$. If $f$ is an s-concave function, so is
$\lambda\times_{p,s} f$ for $\lambda>0$. In addition, if $f, g$ are $s$-concave functions containing the origin in their supports, the function $f \oplus_{p,s}g$ is also $s$-concave and contain the origin in its support, which can be deduced from Lemma \ref{convex} by taking $M=\{(a,b):a^q+b^q=1,a,b\geq0\}$.

\begin{lem}\label{convex}Let $M\subset \mathbb{R}^2$ be a nonempty compact set with nonnegative coordinates and $M\neq \{(0,0)\}$. Let $f,g:\mathbb{R}^n\rightarrow [0,\infty)$ be s-concave functions where $s>0$. Then $f \oplus_{M,s} g$ is s-concave if one of the following conditions holds:

\noindent\emph{(i)}~M is convex.

\noindent\emph{(ii)}~$\operatorname{supp}f$ and $\operatorname{supp}g$ contain the origin.

\end{lem}

\begin{proof}

Set $h=\left[f \oplus_{M,s} g\right].$ Since $M\neq \{(0,0)\}$, we get that
${ h }$ is not identically zero. This gives that $\operatorname{supp} h$ is nonempty.

We turn to prove the compactness of $\operatorname{supp}{ h }$. It is equivalent to proving that $\{z:~{ h }(z)>0\}$ is bounded. By the definition of $h$, we have
\begin{equation}\label{zh(z)subset}
\{z:~{ h }(z)>0\}\subset\operatorname{supp}f\oplus_M \operatorname{supp}g.
\end{equation}
Since $\operatorname{supp}f,~M$~and $\operatorname{supp}g$ are all compact, $\operatorname{supp}f\oplus_M \operatorname{supp}g$ is compact. Therefore, we obtain that $\{z:~{ h }(z)>0\}$ is bounded. Therefore, $\operatorname{supp}h$ is compact.

We will prove that
\begin{equation}\label{suppf+Mg}
\operatorname{supp}h=\operatorname{supp}f\oplus_M \operatorname{supp}g.
\end{equation}
The compactness of $\operatorname{supp}f,~M$~and $\operatorname{supp}g$ gives that $\operatorname{supp}f\oplus_M \operatorname{supp}g$ is the closure of $\{x:f(x)>0\}\oplus_M\{y:g(y)>0\}$. It follows from our assumption of $M$ that
$$\{x:f(x)>0\}\oplus_M\{y:g(y)>0\}\subset\{z:h(z)>0\}.$$
Taking closure on both side gives
$$\operatorname{supp}f\oplus_M \operatorname{supp}g\subset\overline{\{z:h(z)>0\}}=\operatorname{supp}{h}.$$
Now, (\ref{zh(z)subset}) implies (\ref{suppf+Mg}).

Let $z_1,z_2\in \operatorname{supp}{ h }$ and $\theta\in(0,1).$ For given $\varepsilon>0,$ there exist $(a_1,b_1),(a_2,b_2)\in M,x_1,x_2\in\operatorname{supp} f,y_1,y_2\in\operatorname{supp}g$ with
\begin{equation}\label{z1z2}
z_1=a_1x_1+b_1y_1,z_2=a_2x_2+b_2y_2
\end{equation}
such that
\begin{equation}\label{leqf+g}\begin{aligned} h (z_1)^{\frac{1}{s}}-\varepsilon\leq a_1f(x_1)^{\frac{1}{s}}+b_1g(y_1)^{\frac{1}{s}},\\ h (z_2)^{\frac{1}{s}}-\varepsilon\leq a_2f(x_2)^{\frac{1}{s}}+b_2g(y_2)^{\frac{1}{s}}.\end{aligned}\end{equation}

It remains to prove that $\operatorname{supp}{ h }$ is convex and $h^\frac{1}{s}$ is concave in its support. By (\ref{suppf+Mg}), the former is equivalent to proving that there exist $(a,b)\in M, x\in\operatorname{supp} f,y\in\operatorname{supp}g$ such that
\begin{equation}\label{z1z2xy}
  (1-\theta)z_1+\theta z_2=ax+by.
\end{equation}
The latter is equivalent to proving
 \begin{equation}\label{fmsgconcaveinequality} h ((1-\theta)z_1+\theta z_2)^{\frac{1}{s}}
\geq(1-\theta)  h (z_1)^{\frac{1}{s}}+\theta h (z_2)^{\frac{1}{s}}.\end{equation}

(i) Since $\operatorname{supp} f$ and $\operatorname {supp} g$ are convex, let
\begin{equation}\label{xlambdaymu}x=(1-\lambda)x_1+\lambda x_2,y=(1-\mu)y_1+\mu  y_2,\end{equation}
where $\lambda,\mu\in[0,1]$ are to be determined. Then, by (\ref{z1z2}), (\ref{z1z2xy}) becomes
\begin{equation}\label{xy}\left\{\begin{aligned}
(1-\theta)a_1&=(1-\lambda)a,\\
\theta a_2&=\lambda  a,\\
(1-\theta)b_1&=(1-\mu) b,\\
\theta b_2&=\mu  b.
\end{aligned}\right.\end{equation}
This is equvalent to solve the system
$$\left\{\begin{aligned}(1-\theta)a_1+\theta a_2= a,\\
(1-\theta)b_1+\theta b_2=b.\end{aligned}\right.$$
Since $M$ is convex and $(a_1,b_1),(a_2,b_2)\in M$, one can find $(a,b)\in M$ that satisfies this system. Thus, (\ref{z1z2xy}) holds.

Therefore, by (\ref{z1z2xy}), the condition that $a,b\geq 0$, (\ref{xlambdaymu}), (\ref{xy}), (\ref{z1z2}) and (\ref{leqf+g}), we get
$$\begin{aligned}& h ((1-\theta)z_1+\theta z_2)^{\frac{1}{s}}\\\geq &af(x)^{\frac{1}{s}}+bg(y)^{\frac{1}{s}}\\
\geq &a(1-\lambda) f( x_1)^{\frac{1}{s}}+a\lambda f( x_2)^{\frac{1}{s}}+b(1-\mu) f( y_1)^{\frac{1}{s}}+b \mu f(y_2)^{\frac{1}{s}}\\
=&(1-\theta)a_1f( x_1)^{\frac{1}{s}}+(1-\theta)b_1f( y_1)^{\frac{1}{s}}+\theta a_2 f( x_2)^{\frac{1}{s}}+\theta b_2f(y_2)^{\frac{1}{s}}\\
\geq&(1-\theta)  h (z_1)^{\frac{1}{s}}+\theta h (z_2)^{\frac{1}{s}}-\varepsilon.\end{aligned}$$
Since $\varepsilon$ is arbitrary, (\ref{fmsgconcaveinequality}) holds.

(ii) Since  $\operatorname{supp} f$ and $\operatorname {supp} g$ are convex and contain the origin, let
\begin{equation}\label{alphabetaxlambdaymu}x=\alpha\left((1-\lambda)x_1+\lambda x_2\right),y=\beta\left((1-\mu)y_1+\mu  y_2\right),
\end{equation}
where $\lambda,\mu,\alpha,\beta\in[0,1]$ are to be determined. Similarly,
(\ref{z1z2xy}) becomes
\begin{equation}\label{alphabetaxy}\left\{\begin{aligned}
(1-\theta)a_1&=(1-\lambda)\alpha a,\\
\theta a_2&=\lambda \alpha  a,\\
(1-\theta)b_1&=(1-\mu)\beta b,\\
\theta b_2&=\mu\beta  b.
\end{aligned}\right.\end{equation}
This is equvalent to solve the system
$$\left\{\begin{aligned}(1-\theta)a_1+\theta a_2= \alpha a,\\
(1-\theta)b_1+\theta b_2=\beta b.\end{aligned}\right.$$
Set $a=\max\{a_1,a_2\},b=\max\{b_1,b_2\}$, then one can find $\alpha,\beta\in[0,1]$ that satisfy this system. Thus, (\ref{z1z2xy}) holds.

It follows from the fact that $f,g$ are s-concave and the assumption of (ii) that
$$f(\alpha x)^\frac{1}{s}\geq \alpha f(x)^\frac{1}{s}+(1-\alpha)f(o)^\frac{1}{s}\geq \alpha f(x)^\frac{1}{s}, $$
and
$$g(\beta y)^\frac{1}{s}\geq \beta g(y)^\frac{1}{s}+(1- \beta ) g(o)^\frac{1}{s}\geq\beta g(y)^\frac{1}{s},$$
where $o$ denotes the origin of $\mathbb{R}^n.$ Together with (\ref{z1z2xy}), the condition that $a,b\geq 0$, (\ref{alphabetaxlambdaymu}), (\ref{alphabetaxy}), (\ref{z1z2}) and (\ref{leqf+g}), we obtain
$$\begin{aligned}& h ((1-\theta)z_1+\theta z_2)^{\frac{1}{s}}\\\geq &af( x)^{\frac{1}{s}}+bg( y)^{\frac{1}{s}}
\\
\geq &(1-\lambda)a  f( \alpha x_1)^{\frac{1}{s}}+\lambda a f(\alpha  x_2)^{\frac{1}{s}}+(1-\mu) b f(\beta y_1)^{\frac{1}{s}}+\mu b   f(\beta y_2)^{\frac{1}{s}}
\\
\geq &(1-\lambda)a \alpha f( x_1)^{\frac{1}{s}}+\lambda a \alpha f( x_2)^{\frac{1}{s}}+(1-\mu) b\beta f( y_1)^{\frac{1}{s}}+\mu b  \beta f(y_2)^{\frac{1}{s}}\\
=&(1-\theta)a_1f( x_1)^{\frac{1}{s}}+(1-\theta)b_1f( y_1)^{\frac{1}{s}}+\theta a_2 f( x_2)^{\frac{1}{s}}+\theta b_2f(y_2)^{\frac{1}{s}}\\
\geq&(1-\theta)  h (z_1)^{\frac{1}{s}}+\theta h (z_2)^{\frac{1}{s}}-\varepsilon.\end{aligned}$$
Since $\varepsilon$ is arbitrary, (\ref{fmsgconcaveinequality}) holds.\end{proof}

\begin{lem}\label{intkfg} Let $s>0$ be an integer and let $M\subset \mathbb{R}^2$ be a nonempty set with nonnegative coordinates. Then, for any two functions $f,g:\mathbb{R}^n\rightarrow [0,\infty)$ with nonempty supports,
$$\mathcal{K}_f\oplus_M\mathcal{K}_g\subset \mathcal{K}_{f\oplus_{M,s}g}$$
and
$$\operatorname{int}\left(\mathcal{K}_f\oplus_M\mathcal{K}_g\right)=\operatorname{int}\mathcal{K}_{f\oplus_{M,s}g},$$
where $\operatorname{ int} A$ denotes the interior of a set $A$.

\end{lem}

\begin{proof}
First, we prove
\begin{equation}\label{Msubset1}
\mathcal{K}_f\oplus_M\mathcal{K}_g\subset \mathcal{K}_{f\oplus_{M,s}g}.
\end{equation}

Let $(x,x')\in \mathcal{K}_f,(y,y')\in\mathcal{K}_g,(a,b)\in M.$ Then
$$|x'|\leq f(x)^{\frac{1}{s}},|y'|\leq g(y)^{\frac{1}{s}},a,b\geq 0.$$
By the definition of $f\oplus_{M,s}g,$ we obtain
$$\begin{aligned}\left[f\oplus_{M,s}g\right](ax+by)^{\frac{1}{s}}&\geq
af(x)^{\frac{1}{s}}+bg(y^{\frac{1}{s}})\\
&\geq a|x'|+b|y'|\\
&\geq |ax'+by'|.
\end{aligned}$$
That is, $(ax+by,ax'+by')\in \mathcal{K}_{f\oplus_{M,s}g}$. Thus, (\ref{Msubset1}) holds.

It remains to prove
\begin{equation}\label{Msubset2}
\operatorname{int}\mathcal{K}_{f\oplus_{M,s}g}\subset \mathcal{K}_f\oplus_M\mathcal{K}_g.
\end{equation}

Without loss of generality, we can assume that $\operatorname{int}\mathcal{K}_{f\oplus_{M,s}g}$ is nonempty. Let $(z,z')\in \operatorname{int}\mathcal{K}_{f\oplus_{M,s}g}$. Then, there exists $\varepsilon>0,$ such that
$$\left[{f\oplus_{M,s}g}\right](z)^\frac{1}{s}-\varepsilon>|z'|.$$
By the definition of $f\oplus_{M,s}g,$ there exist $(a,b)\in M,x\in\operatorname{supp}f,y\in\operatorname{supp}g$ with $z=ax+by$ such that
$$  af(x)^\frac{1}{s}+bg(y)^\frac{1}{s}>
\left[{f\oplus_{M,s}g}\right](z)^\frac{1}{s}-\varepsilon.$$
Therefore, we get
\begin{equation}\label{Msubset3}
af(x)^\frac{1}{s}+bg(y)^\frac{1}{s}>|z'|.
\end{equation}

Set
$$x'=\frac{f(x)^\frac{1}{s}}{af(x)^\frac{1}{s}+bg(y)^\frac{1}{s}}z',y'=\frac{g(y)^\frac{1}{s}}{af(x)^\frac{1}{s}+bg(y)^\frac{1}{s}}z'.$$
By (\ref{Msubset3}),
$$(x,x')\in \mathcal{K}_f,(y,y')\in\mathcal{K}_g.$$
Thus$$
(z,z')=a(x,x')+b(y,y')\in\mathcal{K}_f\oplus_M\mathcal{K}_g.$$
Therefore, (\ref{Msubset2}) holds.
\end{proof}

\begin{rem}\label{f+pgsubset}Let $p\geq 1$. By
Lemma \ref{intkfg}, we can conclude that for any $\mu,\omega>0$ and integer $s>0$,
$$\mathcal{K}_{[\mu\times_{p,s}f]}+_p\mathcal{K}_{[\omega\times_{p,s}g]}\subset\mathcal{K}_{[\mu\times_{p,s}f]\oplus_{p,s} [\omega \times_{p,s}g]}$$
and
\begin{equation}\label{intKf+g} \operatorname{int} \left(\mathcal{K}_{[\mu\times_{p,s}f]}+_p\mathcal{K}_{[\omega\times_{p,s}g]}\right)=\operatorname{int} \mathcal{K}_{[\mu\times_{p,s}f]\oplus_{p,s} [\omega \times_{p,s}g]}.\end{equation}
\end{rem}

\section{proof of main theorems}

We turn to prove Theorem \ref{LpBMs}.

\begin{proof}[proof of Theorem \ref{LpBMs}]First assume that $s$ is an integer.

The $L_p$ Brunn-Minkowski inequality (\ref{LpBMI}) for $(n+s)$-dimensional sets, (\ref{lambdatimesv}) and Remark \ref{f+pgsubset} implies
$$\begin{aligned}V^\ast_{n+s}(\mathcal{K}_{[\mu\times_{p,s}f]\oplus_{p,s} [\omega \times_{p,s}g]})^\frac{p}{n+s}&\geq V_{n+s}(\mathcal{K}_{[\mu\times_{p,s}f]}+_p\mathcal{K}_{[\omega\times_{p,s}g]})^\frac{p}{n+s}
\\&\geq \mu V_{n+s}(\mathcal{K}_f)^\frac{p}{n+s}+\omega V_{n+s}(\mathcal{K}_g)^\frac{p}{n+s},\end{aligned}$$
where $V^\ast_{n+s}$ stands for outer Lebesgue measure (the set $\mathcal{K}_{[\mu\times_{p,s}f]\oplus_{p,s} [\omega \times_{p,s}g]}$ may be non-measurable).
By (\ref{Kfvolume}), this is equivalent to
\begin{equation}\label{Lpf+pg}\left(\int^{\ast}_{\mathbb{R}^n}[\mu\times_{p,s}f]\oplus_{p,s}[\omega\times_{p,s}g]\right)^{\frac{p}{n+s}}\geq \mu\left(\int_{\mathbb{R}^n}f\right)^\frac{p}{n+s}+\omega\left(\int_{\mathbb{R}^n}g\right)^\frac{p}{n+s},\end{equation}
where $\int^\ast$ is the outer integral. Note that $\int h=\int^{*} h$. Thus, it follows from (\ref{hgeqf+g}) and (\ref{Lpf+pg}) that (\ref{FLpBMIconclusion}) holds.

Next assume that $s=\frac{l}{t}$ is rational.

Note that, by H\"{o}lder's
inequality (See \cite{MR0046395}) and (\ref{FLpBMIcondition}), for any $x_1, \cdots, x_t, y_1, \cdots, y_t \in\mathbb{R}^n$,
\begin{equation}\label{sintegeri}
\begin{aligned}
(1-\lambda)^\frac{1}{q} \mu^\frac{1}{p}\prod_{i=1}^{t} f\left(x_{i}\right)^{\frac{1}{t{s}}}+\lambda^\frac{1}{q}\omega^\frac{1}{p} \prod_{i=1}^{t} g\left(y_{i}\right)^{\frac{1}{t{s}}} & \leq\left(\prod_{i=1}^{t}\left((1-\lambda)^\frac{1}{q}\mu ^\frac{1}{p} f\left(x_{i}\right)^{\frac{1}{s}}+\lambda^\frac{1}{q}\omega ^\frac{1}{p} g\left(y_{i}\right)^{\frac{1}{s}}\right)\right)^{\frac{1}{t}} \\
& \leq \prod_{i=1}^{t} h\left((1-\lambda)^\frac{1}{q}\mu^\frac{1}{p} x_{i}+\lambda^\frac{1}{q}\omega^\frac{1}{p} y_{i}\right)^{\frac{1}{t{s}}}.
\end{aligned}
\end{equation}
For a function $r : \mathbb{R}^n \rightarrow [0,\infty),$ we define $\tilde{r} : \mathbb{R}^{nt}\rightarrow[0,\infty)$ by $$\tilde{r}(x)=\tilde{r}\left(x_{1}, \ldots, x_{t}\right)=\prod_{i=1}^{t} r\left(x_{i}\right)$$ where $x = (x_1, \cdots, x_t) \in (\mathbb{R}^n)^t$ are
coordinates in $\mathbb{R}^{nt}$. Thus, (\ref{sintegeri}) implies that for $x, y \in\mathbb{R}^{nt}$,
$$\tilde{h}((1-\lambda)^{\frac{1}{q}}\mu^\frac{1}{p} x+\lambda^{\frac{1}{q}}\omega^\frac{1}{p} y)^{\frac{1}{ts}}\geq (1-\lambda)^{\frac{1}{q}}\mu^\frac{1}{p} \tilde{f}(x)^{\frac{1}{ts}}+\lambda^{\frac{1}{q}}\omega^\frac{1}{p} \tilde{g}(y)^{\frac{1}{ts}}.$$

Now, $ts=l$ is integer. This gives that
$$
\begin{aligned}
\left(\int h\right)^{\frac{p}{n+s}} &=\left(\int \tilde{h}\right)^{\frac{p}{t(n+s)}} \geq \mu\left(\int \tilde{f}\right)^{\frac{p}{t(n+s)}}+\omega\left(\int \tilde{g}\right)^{\frac{p}{t(n+s)}} \\
&=\mu\left(\int f\right)^{\frac{p}{n+s}}+\omega\left(\int  g\right)^{\frac{p}{n+s}}.
\end{aligned}
$$
The case that $s$ is irrational follows by a standard approximation argument.
\end{proof}

Using Theorem \ref{LpBMs}, we will prove a functional $L_p$
Minkowski inequality.

For $p\geq 1,s>0$ and two functions $f, g :
\mathbb{R}^n \rightarrow [0,\infty)$ with nonempty support, we define
$$
\tilde{\mathcal{S}}_{p,s}(f ; g)=\frac{p}{n+s} \lim _{\varepsilon \rightarrow 0^{+}} \frac{\int \left[f \oplus_{p,s}\left(\varepsilon \times_{p,s} g\right)\right]-\int {}f}{\varepsilon}
$$
whenever the integrals are defined and the limit exists. The motivation of this definition is from the defintion of $L_p$ mixed volume, see \cite{MR1231704}. When $f$ and $g$ is $s$-concave, this limit exists, see the proof in the following Corollary.

\begin{cor}Let $s > 0$ and $f, g :
\mathbb{R}^n \rightarrow [0,\infty)$ be integrable functions with nonempty support
such that $
\tilde{\mathcal{S}}_{p,s}(f ; g)$ exists. Then,
\begin{equation}\label{lpfunctionmix}
\tilde{\mathcal{S}}_{p,s}(f ; g) \geq\left(\int f\right)^{1-\frac{p}{n+s}}\left(\int  g\right)^{\frac{p}{n+s}}.
\end{equation}
If $s$ is an integer, $f = \lambda \times_{p,s} g,$~where$~\lambda> 0$, and $g$ is s-concave such that $\operatorname{supp} g$ has nonempty interior and contains the origin, then equality holds.
\end{cor}

\begin{proof} By (\ref{Lpf+pg}),
\begin{align}\label{lpfge}
\int \left[f \oplus_{p,s}\left(\varepsilon \times_{p,s} g\right)\right] & \geq\left(\left(\int  f\right)^{\frac{p}{n+s}}+\varepsilon\left(\int g\right)^{\frac{p}{n+s}}\right)^\frac{n+s}{p} \\
& \geq\left(\int  f\right)+\varepsilon\cdot\frac{n+s}{p}\left(\int f\right)^{1-\frac{p}{n+s}}\left(\int g\right)^{\frac{p}{n+s}}.\notag
\end{align}
Since $
\tilde{\mathcal{S}}_{p,s}(f ; g)$ exists, the definition of $
\tilde{\mathcal{S}}_{p,s}(f ; g)$ implies the desired inequality.

If $s$ is an integer, $f = \lambda \times_{p,s} g$ and $g$ is s-concave with $\operatorname{supp} g$ such that $\operatorname{supp} g$ has nonempty interior and contains the origin, then $\mathcal{K}_f$ and $\mathcal{K}_{[\varepsilon\times_{p,s}g]}$ are convex bodies containing the origin for $\varepsilon>0$, and $\mathcal{K}_f=\lambda^\frac{1}{p}\mathcal{K}_g$.

 By (\ref{lpsumconvexbody}) and the homogeneity of support function,
$$h_{\mathcal{K}_f+_p\mathcal{K}_{[\varepsilon\times_{p,s}g]}}^p(x)=h_{\mathcal{K}_f}^p(x)+_ph_{\mathcal{K}_{[\varepsilon\times_{p,s}g]}}^p(x)
=h_{\lambda^\frac{1}{p}\mathcal{K}_g}^p(x)+h_{\varepsilon^\frac{1}{p}\mathcal{K}_g}^p(x)=(\lambda+\varepsilon)h^p_{\mathcal{K}_g}(x)$$
for all $x \in\mathbb{R}^n$, which shows that
$\mathcal{K}_f+_p\mathcal{K}_{[\varepsilon\times_{p,s}g]}=(\lambda+\varepsilon)^\frac{1}{p}\mathcal{K}_g$ is a convex body.
It follows from Remark \ref{f+pgsubset} that
$$\operatorname{int} \mathcal{K}_{f\oplus_{p,s} [\varepsilon \times_{p,s}g]}= \operatorname{int} \left(\mathcal{K}_{f}+_p\mathcal{K}_{[\varepsilon\times_{p,s}g]}\right)
=\operatorname {int}\left((\lambda+\varepsilon)^\frac{1}{p} \mathcal{K}_g)\right).$$
Since these sets are convex, we get that
$$V_{n+s}\left(\mathcal{K}_{f\oplus_{p,s} [\varepsilon \times_{p,s}g]}\right)=(\lambda+\varepsilon)^\frac{n+s}{p}V_{n+s}\left(\mathcal{K}_{g}\right).$$
Therefore, by (\ref{Kfvolume}),
$$\begin{aligned}\int \left[f \oplus_{p,s}\left(\varepsilon \times_{p,s} g\right)\right]-\int  f&=\frac{V_{n+s}\left(\mathcal{K}_{f\oplus_{p,s} [\varepsilon \times_{p,s}g]}\right)-V_{n+s}(\mathcal{K}_f)}{\kappa_s}\\
&=\left((\lambda+\varepsilon)^\frac{n+s}{p}-\lambda^\frac{n+s}{p}\right)\frac{V_{n+s}(\mathcal{K}_g)}{\kappa_s}\\
&=\left((\lambda+\varepsilon)^\frac{n+s}{p}-\lambda^\frac{n+s}{p}\right)\int g.\end{aligned}$$
Now, by the definition $\tilde{\mathcal{S}}_{p,s}(f ; g)$ and the condition that $f = \lambda \times_{p,s} g=\lambda ^\frac{s}{p}f(\frac{x}{\lambda^p})$, we have
 $$\begin{aligned}\tilde{\mathcal{S}}_{p,s}(f ; g)&= \frac{p}{n+s}\lim_{\varepsilon\rightarrow 0^+}\frac{(\lambda+\varepsilon)^\frac{n+s}{p}-\varepsilon^\frac{n+s}{p}}{\varepsilon}\int g\\
 &=\lambda^{\frac{n+s}{p}-1}\int  g=\left(\int f\right)^{1-\frac{p}{n+s}}\left(\int  g\right)^{\frac{p}{n+s}}.\end{aligned}$$

\end{proof}

\end{document}